\documentclass[12pt,a4paper]{amsart}

\usepackage{amsmath}
\usepackage{amsfonts}
\usepackage{amssymb}

\usepackage{tikz}
\usepackage{enumitem}
\newcommand*{\circled}[1]{\lower.7ex\hbox{\tikz\draw (0pt, 0pt)%
    circle (.5em) node {\makebox[1em][c]{\small #1}};}}

\newtheorem{theorem}{Theorem}

\newtheorem{lemma}{Lemma}

\newtheorem{corollary}{Corollary}

\theoremstyle{definition}

\newtheorem{remark}{Remark}
\numberwithin{equation}{section}
\begin{document}
\title[Linear operators on Bergman spaces]%
{A class of linear operators on Bergman spaces}

\author{Zengjian Lou*}
\address{Department of Mathematics, Shantou University, Shantou, Guangdong,
515063, \newline P.~R.~China}
\email{zjlou@stu.edu.cn}

\author{Antti Rasila*}
\address{Department of Mathematics with Computer Science,
Guangdong Technion,
Shantou, Guangdong, 515063, \newline P. R. China  \newline
and Department of Mathematics, Technion -- Israel Institute of Technology, Haifa 32000, \newline
Israel.
}
\email{antti.rasila@iki.fi; antti.rasila@gtiit.edu.cn}

\author{Senhua Zhu}
\address{Department of Mathematics,
Shantou University,
Shantou, Guangdong, 515063, \newline P. R. China}
\email{shzhu@stu.edu.cn}

\begin{abstract}
We study the boundedness of the linear operator $S$ on $L^{p}_{a}(dA_{\alpha})~(0<p<\infty)$. In particular, we obtain a sufficient and necessary condition for the compactness of the linear operator $S$ on $L^{p}_{a}(dA_{\alpha})~(1<p<\infty)$. Our results weaken the assumptions of earlier results of J. Miao and D. Zheng in a certain sense.
\end{abstract}

\thanks{
The research was supported by NNSF of China (No. 12471093, 12001082), NSF of Guangdong Province (No. 2024A1515010468, 2024A1515010467) and LKSF STU-GTIIT
Joint Research Grant (No. 2024LKSFG06).}
 \thanks{* Corresponding authors.}

\maketitle

\section{Introduction}

Let $\mathbb{D}$ denote the open unit disk in the complex plane $\mathbb{C}$, and let $dA$ denote the normalized area measure on $\mathbb{D}$.
For $0< p<\infty$, the Bergman space $L_a^p(dA_\alpha)~(\alpha>-1)$ is the space consisting of all analytic functions on $\mathbb{D}$ that belong to $L^p(dA_\alpha)$, where $dA_\alpha(z)=(\alpha+1)(1-|z|^2)^\alpha dA(z)$.
For $1\leq p<\infty$, $L_a^p(dA_\alpha)$ is a Banach space with the norm $\|\cdot\|_{p,\alpha}$ defined as follows:
$$\|f\|_{p,\alpha}=\Big(\int_{\mathbb{D}}|f(z)|^{p}\, dA_{\alpha}(z)\Big)^{\frac{1}{p}}.$$
It is well-known that $L^2_a(dA_\alpha)$ is a Hilbert space with the inner product defined by
$$\langle f,g\rangle=\int_\mathbb{D}f(z)\overline{g(z)}\,dA_\alpha(z),~f,g\in L^2_a(dA_\alpha).$$

For $1<p<\infty$ and $\frac{1}{p}+\frac{1}{q}=1$, we define the integral pairing
$$\langle f,g\rangle=\int_{\mathbb{D}}f(z)\overline{g(z)}\, dA_\alpha(z),$$ where $f\in L_a^p(dA_\alpha)$ and $g\in L_a^q(dA_\alpha)$.

The function
$$K(w,z)=K_{z}(w)=\frac{1}{(1-\overline{z}w)^{2+\alpha}}$$
 is called the reproducing kernel in $L_a^2(dA_\alpha)$.
For each $z \in \mathbb{D}$, the function $k_{z}(w)$ defined by
$$k_{z}(w)=(1-|z|^{2})^{\frac{2+\alpha}{2}}K_{z}(w)$$
is called the normalized reproducing kernel at $z$. It is also in $L^p_a(dA_\alpha)$ for $1\leq p<\infty$.

An interesting problem is to consider the boundedness and compactness of linear operators on Bergman spaces.
By a variety of methods, some  remarkable work has been done on boundedness and compactness of Toeplitz operators with variable symbols, see \cite{Sheldon1,Sheldon,Daniel,zheng}.
Moreover, a list of open problems about the boundedness, compactness and Fredholm properties of Toeplitz operators have been summarized in \cite{Antti}.

Recently, Mitkovski and Wick \cite{Mishko} provided a sufficient condition for the boundedness of general operator $T$ on reproducing kernel Hilbert spaces $H$, which consist of holomorphic functions in $L^2(\Omega,d\mu)$, where $\Omega$ is a domain in $\mathbb{C}^n$ and $\mu$ is a certain finite measure on $\Omega$. Furthermore, under the aforementioned boundedness assumptions, the operator $T$ is compact on $H$ if and only if $\|Tk_z\|\rightarrow 0$  as $d(z,0) \rightarrow \infty$.

For a bounded linear operator $S$ on $L_a^p(dA_\alpha)~(0<p<\infty)$, the operator $S_{z}$ is defined by $S_{z}=U_{z}SU_{z}$, where $U_z$ is a bounded operator on $L_a^p(dA_\alpha)$ defined by
$$U_{z}f=\big(f\circ\varphi_{z}\big) \cdot k_{z}. $$
For $z\in\mathbb{D}$, let $\varphi_{z}$ be the M\"obius transformation  on $\mathbb{D}$ defined by
$$\varphi_{z}(w)=\frac{z-w}{1-\overline{z}w},~ w \in \mathbb{D}.$$
Furthermore, $U_z$ is a unitary operator on $L_a^2(dA_\alpha)$.

In 2004, Miao and Zheng \cite{Jie} showed that if $S$ is a bounded operator on $L_a^p(dA)~(1<p<\infty)$ such that
 $$\sup_{z\in \mathbb{D}}\|S_{z}1\|_{m}<\infty,~ \sup_{z\in \mathbb{D}}\|S_{z}^{\ast}1\|_{m}<\infty,$$
for some $m>3/(p-1)$,
 then $S$ is compact on $L_a^p(dA)~(1<p<\infty)$ if and only if $\widetilde{S}(z) \rightarrow 0$ as $z \rightarrow \partial \mathbb{D}$. Here, $\widetilde{S}$ denotes the Berezin transform of $S$, defined by
$$\widetilde{S}(z)=\langle Sk_{z},k_{z}\rangle,~ z\in \mathbb{D}.$$

In particular, the result for $p=2$ agrees with the special case $\Omega=\mathbb{D}$ result of Mitkovski and Wick \cite{Mishko}. Similar results on Fock spaces have also been studied by the authors of this paper in \cite{Lou} (see also \cite{Xiao}).

In this paper, we further explore the boundedness and compactness of linear operators $S$ on the Bergman space $L_a^p(dA_\alpha)$ using a different method.
We first establish a sufficient condition for the boundedness of a linear operator $S$ on $L^{p}_{a}(dA_{\alpha})$.
We then derive the compactness of $S$ on $L^{p}_{a}(dA_{\alpha})$, where the requirement of $\|S^{\ast}_{z}1\|_{m}$ is dropped, in contrast to the result in \cite{Jie} (see Remark \ref{rem1} for details). It is well known that linear combinations of reproducing kernels are dense in  Bergman spaces. From the definition of $S_{z}$, we have
$$(S_{z}1)(w)=k_{z}(\varphi_{z}(w)) \cdot k_{z}(w).$$
This indicates that if $S$ is well defined on reproducing kernels, then $S_{z}1$ is well defined.

Throughout this paper, we assume that the domain of $S$ contains all reproducing kernels. Our main results are as follows:

\newtheorem*{them1}{Theorem A}
 \begin{them1}
Suppose that $1<p<\infty$, $\alpha>-1$ and $S$ is a linear operator on $L^{p}_{a}(dA_{\alpha})$. If there exist a positive constant $C$ and some $m>p\frac{2+\alpha}{1+\alpha}\max\{1,\frac{1}{p-1}\}$, such that $\|S_{z}1\|_{m,\alpha}\leq C$ for all $z \in \mathbb{D}$, then $S$ is bounded.
\end{them1}

\newtheorem*{them3}{Theorem B}
\begin{them3}\label{B1}
Suppose that $1<p<\infty$, $\alpha>-1$ and $S$ is a linear operator on $L_{a}^{p}(dA_{\alpha})$. If there exist a positive constant $C$ and some $m>p\frac{2+\alpha}{1+\alpha} \max\{1,\frac{1}{p-1}\}$, such that $\|S_{z}1\|_{m,\alpha}\leq C$ for all $z \in \mathbb{D}$, then $S$ is compact if and only if $\widetilde{S}(z) \rightarrow 0$ as $z \rightarrow \partial \mathbb{D}$.
\end{them3}


\begin{remark}
\label{rem1}


 Theorem B weakens the assumption of Theorem 1.1 in \cite{Jie}.
Specifically, the condition $\sup_{z\in \mathbb{D}}\|S_{z}^{\ast}1\|_{m}<\infty$ has been removed in Theorem B.
On the other hand, when $\alpha=0$ and $1<p<\frac{3}{2}$, it is easy to verify that
$$p\frac{2+\alpha}{1+\alpha} \max \Big\{1,\frac{1}{p-1}\Big\}< \frac{3}{p-1}.$$
Thus, our requirement for $m$ is weaker than that in \cite{Jie}, because $\|f\|_{p} \leq \|f\|_{q}$ whenever $p<q$.


\end{remark}

\section{Bounded linear operators}

In this section, we provide sufficient conditions for the boundedness of linear operators on Bergman spaces $L^{p}_{a}(dA_\alpha)$.

The {\it Bergman disk} is defined by $D(a,r)=\{z\in \mathbb{D}:~ \beta(z,a)<r\}$
where	
$$
\beta(z,w)=\frac{1}{2}\log\frac{1+\rho(z,w)}{1-\rho(z,w)}
$$
is the {\it Bergman metric}, and
$$
\rho(z,w)=\frac{|z-w|}{|1-z\overline{w}|}. 
$$
For  $r>0 $ and $a_k\in\mathbb{D}$, we say that $\{a_k\} $ is a {\it  $r$-lattice} if
$$
\mathbb{D}=\bigcup_kD(a_k,r); \ \ \
       \beta(a_i,a_j)\geq t/2  \  \text{ for all} \  i, j \ \text{with}\  i\neq j.
$$

The following lemmas are used in our proofs:


\begin{lemma}\cite[Theorem 4.33]{zhu}\label{lemy2}
Suppose $p>0$, $\alpha>-1$, and $$b>\max\Big\{1,\frac{1}{p}\Big\}+\frac{\alpha+1}{p}.$$
Then there exists a constant $\sigma>0$ such that for any $r$-lattice $\{a_{k}\}$ in the Bergman metric, where $0<r<\sigma$, the space $L^{p}_{a}(dA_{\alpha})$ consists exactly of functions of the form
\begin{equation}\label{formula1}
  f(z)=\sum_{k=1}^{\infty}c_{k}\frac{(1-|a_{k}|^{2})^{\frac{pb-2-\alpha}{p}}}{(1-z\overline{a_{k}})^{b}},
\end{equation}
where $\{c_{k}\} \in l^{p}$. The series in $(\ref{formula1})$ converges in the norm of  $L^{p}_{a}(dA_{\alpha})$, and the norm of $f$ in $L^{p}_{a}(dA_{\alpha})$, $\|f\|_{p,\alpha}$, is comparable to
$$\|f\|=\inf\Big\{\Big[\sum_{k=1}^{\infty}|c_{k}|^{p}\Big]^{\frac{1}{p}}:\{c_{k}\} ~ \text{satisfies} ~(\ref{formula1})\Big\}.$$
\end{lemma}

%
\begin{lemma}\cite[Theorem 1.3]{Liuc}\label{lemy15}
Suppose $z\in \mathbb{D}$, $c$ is real, $t>-1$, and $$I_{c,t}(z)=\int_{\mathbb{D}}\frac{(1-|w|^{2})^{t}}{|1-z\overline{w}|^{2+t+c}}\, dA(w).$$
\begin{itemize}
  \item [(a)] If $c<0$, then for all $z\in \mathbb{D}$,
$$\frac{\Gamma(1+t)}{\Gamma(2+t)}\leq I_{c,t}(z)\leq \frac{\Gamma(1+t)\Gamma(-c)}{\Gamma\left(\frac{2+t-c}{2}\right)^{2}}.$$
  \item [(b)] If $c>0$, then for all $z\in \mathbb{D}$,  $$\frac{\Gamma(1+t)}{\Gamma(2+t)}\leq (1-|z|^{2})^{c} I_{c,t}(z)\leq \frac{\Gamma(1+t)\Gamma(c)}{\Gamma\left(\frac{2+t+c}{2}\right)^{2}}.$$
  \item [(c)] If $c=0$, then for all $z\in \mathbb{D}$,
$$\frac{\Gamma(1+t)}{\Gamma(1+\frac{t}{2})}\leq |z|^{2}\Big(\log\frac{1}{1-|z|^{2}}\Big)^{-1} I_{c,t}(z)\leq \frac{1}{1+t}.$$
\end{itemize}
\end{lemma}

\begin{lemma}\label{ly1}
 Let $t_{1}>1$ and
 $$L(w)=\sum_{k=1}^{\infty}\frac{(1-|a_{k}|^{2})^{t_{1}}}{|1-\overline{a_{k}}w|^{t_{2}}}, $$
where $\{a_{k}\}$ is an $r$-lattice on $\mathbb{D}$.
Then
\begin{equation}\label{formula10}
L(w) \leq \left\{
                   \begin{array}{ll}
                     \frac{C_{t_{1},t_{2},r}}{(1-|w|^{2})^{t_{2}-t_{1}}},&t_{2}>t_{1}, \quad w \in  \mathbb{D},\\
                    C_{t_{1},t_{2},r}\log\frac{1}{1-|z|^{2}},  &t_{2}=t_{1}, \quad w \in  \mathbb{D},\\
                    C_{t_{1},t_{2},r}, &t_{2}<t_{1}.\\
                   \end{array}
                 \right.
\end{equation}
where $C_{t_{1},t_{2},r}$ is a positive constant depending on $t_1,~t_2,~r$.
\end{lemma}
\begin{proof}
By applying Propositions 4.4 and  4.5 of \cite{zhu}, we have
\begin{equation}\label{formula7}
  |D(a_{k},r)|= \frac{(1-|a_{k}|^{2})^{2}s^{2}}{(1-|a_{k}|^{2}s^{2})^{2}},
\end{equation}
where $s=\tanh r$.
By applying Proposition 4.5 in \cite{zhu} again, we observe that
$$\frac{1-s}{1+s}(1-|v|^2)\leq 1-|a_k|^2\leq \frac{1+s}{1-s}(1-|v|^2),~v\in D(a_k,r).$$
So
\begin{equation}\label{formula8}
  C_{t_{1},r}^{-1}(1-|v|^{2})^{t_{1}-2}\leq(1-|a_{k}|^{2})^{t_{1}-2}\leq C_{t_{1},r}(1-|v|^{2})^{t_{1}-2}, \quad v \in D(a_{k},r),
\end{equation}
 where $$C_{t_{1},r}=\left\{
                   \begin{array}{ll}
                    (\frac{1-s}{1+s})^{t_{1}-2}, \quad & t_{1}<2; \\
                     (\frac{1-s}{1+s})^{2-t_{1}}, \quad & t_{1}\geq 2.
                   \end{array}
                 \right.
 $$

 According to Lemma 4.30 in \cite{zhu}, there exists a positive constant $C_{t_{2},r}$ dependent on $r$ and $t_{2}$, such that for all $w\in\mathbb{D}$,
\begin{equation}\label{formula9}
 |1-\overline{v}w|^{t_{2}} \leq C_{t_{2},r}|1-\overline{a_k}w|^{t_{2}}, \quad v \in D(a_{k},r).
\end{equation}
By combining (\ref{formula7}), (\ref{formula8}) and (\ref{formula9}), we obtain
\begin{equation}\nonumber
  \begin{split}
   \frac{(1-|a_{k}|^{2})^{t_{1}}}{|1-\overline{a_{k}}w|^{t_{2}}}
   &=\frac{(1-|a_{k}|^{2})^{t_{1}-2}(1-s^{2}|a_{k}|^{2})^{2}}{s^{2}|1-\overline{a_{k}}w|^{t_2}}|D(a_{k},r)|\\
   &\leq C_{t_{1},r}C_{t_{2},r} \frac{(1+s^{2})^{2}}{s^{2}} \int_{D(a_{k},r)} \frac{(1-|v|^{2})^{t_{1}-2}}{|1-\overline{v}w|^{t_{2}}}dA(v).
  \end{split}
\end{equation}

By Lemma 4.7 in \cite{zhu} and Lemma \ref{lemy15}, there exists a positive constant $C_{t_{1},t_{2},r}$ dependent on $t_{1}$, $t_{2}$ and $r$, such that
\begin{equation}\nonumber
\begin{split}
\sum_{k=1}^{\infty}\frac{(1-|a_{k}|^{2})^{t_{1}}}{|1-\overline{a_{k}}w|^{t_{2}}}
&\leq C_{t_{1},t_{2},r} \int_{\mathbb{D}} \frac{(1-|v|^{2})^{t_{1}-2}}{|1-\overline{v}w|^{ 2+(t_{1}-2)+(t_{2}-t_{1})}}\, dA(v),\\
\end{split}
\end{equation}
which yields the desired result.
\end{proof}

\begin{lemma} \cite[p. 97]{zhu}\label{lemy9}
For $1<p<\infty$, $\alpha>-1$, $f_{n} \rightarrow f$ weakly in $L^{p}_{a}(dA_{\alpha})$ if and only if $\|f_{n}\|_{p,\alpha}$ is bounded and $f_{n}(z)\rightarrow f(z)$ uniformly on each compact set of $\mathbb{D}$.
\end{lemma}


\begin{lemma} \cite[Theorem 1.1 of Chapter VI]{conway}\label{lemy10}
If $\mathrm{X}$ and $\mathrm{Y}$ are Banach spaces and $T:\mathrm{X} \rightarrow \mathrm{Y}$ is a linear transformation, then the following statements are equivalent.
\begin{itemize}
  \item [(a)] $T $ is bounded;
  \item [(b)] $T^{\ast}(\mathrm{Y}^{\ast})\subseteq \mathrm{X}^{\ast}$;
  \item [(c)] $T:(\mathrm{X},weak)\rightarrow (\mathrm{Y},weak)$ is continuous.
\end{itemize}
\end{lemma}

The first part of the following lemma can be obtained from the proof of Theorem 4.33 of \cite{zhu}. For the completeness of the paper, we give a proof here.

\begin{lemma}\label{lem2}
Suppose that $1< p <\infty$, $\alpha>-1$, and $\sigma>0$ such that for any $r$-lattice $\{a_{k}\}$ in the Bergman metric, where $0<r<\sigma$. Then we can represent $f \in L^{p}_{a}(dA_{\alpha})$ as
$$f(z)=\sum_{k=1}^{\infty} c_{k} \frac{(1-|a_{k}|^{2})^{\frac{(p-1)(2+\alpha)}{p}}}{(1-\overline{a_{k}}z)^{2+\alpha}},$$
in such a way that the series converges in the norm topology of $L^{p}_{a}(dA_{\alpha})$ and
$$\sum_{k=1}^{\infty} |c_{k}|^{p} \leq C\|f\|_{p,\alpha}^{p}$$
for some positive constant $C$ independent of $f$. Furthermore, if $\{f_{n}\} \rightarrow 0$ weakly in $L^{p}_{a}(dA_{\alpha})$, then
\begin{equation}\label{f2}
  \sum_{|a_{k}|\leq R}|c_{k,n}|^{p} \rightarrow 0~~ \text{as}~~ n \rightarrow \infty
\end{equation}
for any fixed $R<1$, where $c_{k,n}$ are coefficients of
$$f_{n}(z)=\sum_{k=1}^{\infty} c_{k,n} \frac{(1-|a_{k}|^{2})^{\frac{(p-1)(2+\alpha)}{p}}}{(1-\overline{a_{k}}z)^{2+\alpha}}.$$
\end{lemma}

\begin{proof}
According to Lemma \ref{lemy2}, every $f \in L^{p}_{a}(dA_{\alpha})$ can be expressed as
\begin{equation}\label{1}
  f(z)=\sum_{k=1}^{\infty} c_{k} \frac{(1-|a_{k}|^{2})^{\frac{(p-1)(2+\alpha)}{p}}}{(1-\overline{a_{k}}z)^{2+\alpha}},
\end{equation}
where $\{c_k\}\in l^p$, $\{a_k\}$ is an $r$-lattice in the Bergman metric and the series in (\ref{1}) converges in the norm of $L^{p}_{a}(dA_{\alpha})$.
Define an operator $T$ as follows:
$$Tf(z)=\sum_{k=1}^{\infty}\frac{A_{\alpha}(D_{k})f(a_{k})}{(1-\overline{a_{k}}z)^{2+\alpha}},~~f \in L^{p}_{a}(dA_{\alpha}),$$
where $\{D_k\}$ is a decomposition of $\mathbb{D}$ according to Lemma 4.10 in \cite{zhu}.
From the proof of Lemma \ref{lemy2}, it follows that $T$ is bounded and invertible. Let $g=T^{-1}f$, then
\begin{equation}\nonumber
  c_{k}=\frac{A_{\alpha}(D_{k})g(a_{k})}{(1-|a_{k}|^{2})^{\frac{(p-1)(2+\alpha)}{p}}}.
\end{equation}
By Proposition 4.5 in \cite{zhu}, it is obvious that $$A_{\alpha}(D(z,r))\sim (1-|z|^{2})^{2+\alpha}$$
for any $\alpha $ and $r>0$. Thus, there exists a positive constant $C_{r,\alpha}$ dependent only on $r$ and $\alpha$ such that
\begin{equation}\nonumber
  \begin{split}
    A_{\alpha}(D_{k})
    &\leq C_{r,\alpha} (1-|a_{k}|^{2})^{2+\alpha}.\\
  \end{split}
\end{equation}
Thus, by \cite[Proposition 4.18]{zhu},
\begin{equation}\label{formula18}
  \begin{split}
    \sum_{k=1 }^{\infty} |c_{k}|^{p}
    &\leq C_{r,\alpha}^{p-1}\sum_{k=1 }^{\infty} A_{\alpha}(D_{k})|g(a_{k})|^{p}\\
   &\leq C_{r,\alpha}^{p-1}M \int_{\mathbb{C}}|g(w)|^{p}dA_{\alpha}(w)\\
   & \leq  C_{r,\alpha}^{p-1}M \|T^{-1}f\|_{p,\alpha}^{p}\\
   & \leq C\|f\|_{p,\alpha}^{p},
  \end{split}
\end{equation}
where $M$ is the maximal number of this $r$-lattice (see Lemma 4.7 in \cite{zhu}).

Next, we show that $(\ref{f2})$ holds. For any fixed $R\leq 1$, we note that the hyperbolic disk $D(a_{k},r)$ is a Euclidean disk with center and radius given by
$$c_{0}=\frac{1-s^{2}}{1-s^{2}|a_{k}|^{2}}a_{k}, \quad r_{0}=\frac{1-|a_{k}|^{2}}{1-s^{2}|a_{k}|^{2}}s,$$
where $s=\tanh r\in (0,1)$.
Since
\begin{equation}\nonumber
\begin{split}
   |c_{0}|+r_{0}
   =\frac{s+|a_{k}|}{1+s|a_{k}|},
\end{split}
\end{equation}
and the function
$$d(x)=\frac{s+x}{1+sx}$$
is in $x$ for $s\in(0,1)$, it follows that
$$\frac{s+|a_{k}|}{1+s|a_{k}|}\leq \frac{s+R}{1+sR}<1,~\text{when}~0\leq |a_{k}|\leq R<1.$$
Moreover, we have
 $$\bigcup_{|a_{k}|\leq R}D_{k} \subset \bigcup_{|a_{k}|\leq R}D(a_{k},r)\subset \overline{B\big(0,(s+R)/(1+sR)\big)},$$
where $B\big(0,(s+R)/(1+sR)\big)$ is a Euclidean disk, and its closure is a compact set in $\mathbb{D}$.
Since $f_{n} \rightarrow 0$ weakly, by Lemma \ref{lemy9}, $\{f_{n}\}$ is bounded and $f_{n}(z)\rightarrow 0$ uniformly on each compact subset of $\mathbb{D}$. According to Lemma \ref{lemy2}, we have
$$f_{n}(w)=\sum_{k=1}^{\infty} c_{k,n} \frac{(1-|a_{k}|^{2})^{\frac{(p-1)(2+\alpha)}{p}}}{(1-\overline{a_{k}}w)^{2+\alpha}}.$$

 Let $g_{n}=T^{-1}f_{n}$. Lemma \ref{lemy9} and Lemma \ref{lemy10} show that $g_{n}(w) \rightarrow 0$ uniformly on $\overline{B\big(0,(s+R)/(1+sR)\big)}$.
 In other words, for any $\varepsilon>0$, there exists a positive constant $N_{\varepsilon, R}$ such that
\begin{equation}\label{formula26}
  |g_{n}(w)| <\varepsilon,
\end{equation}
for $n>N_{\varepsilon, R}$ and $w \in \overline{B\big(0,(s+R)/(1+sR)\big)}$.
From (\ref{formula18}), we get
\begin{equation}\nonumber
  \begin{split}
      \sum_{|a_{k}|\leq R} |c_{k,n}|^{p}
    &\leq C_{r,\alpha}^{p-1}M \int_{\bigcup_{|a_{k}|\leq R}D_{k}}|g_{n}(w)|^{p}dA_{\alpha}(w)\\
    & \leq C_{r,\alpha}^{p-1}M \int_{\overline{B(0,\frac{s+R}{1+sR})}}|g_{n}(w)|^{p}dA_{\alpha}(w)\\
    & \leq  C_{r,\alpha}^{p-1} M\varepsilon^{p}.
  \end{split}
\end{equation}
We get the desired result.
\end{proof}

\begin{lemma}\label{lemy16}
Suppose $m>0$ and $\alpha>-1$.
If $\|S_{z}1\|_{m,\alpha}\leq C$ for some positive constants $C$ and all $z \in \mathbb{D}$.
Then for every $w \in \mathbb{D}$,
\begin{equation}\label{formula3}
  \begin{split}
   |SK_{z}(w)|\leq \|S_{z}1\|_{m,\alpha} \frac{(1-|z|^{2})^{-\frac{2+\alpha}{m}}(1-|w|^{2})^{-\frac{2+\alpha}{m}}}{|1-\overline{z}w|^{(1-\frac{2}{m})(2+\alpha)}}.
  \end{split}
\end{equation}
\end{lemma}

\begin{proof}
We note that $S_{z}1(w)=(Sk_{z})(\varphi_{z}(w))  k_{z}(w)$,
so by Theorem 4.14 in \cite{zhu}, we get
\begin{equation}\nonumber
  |(Sk_{z})(\varphi_{z}(w))k_{z}(w)| \leq \frac{\|S_{z}1\|_{m,\alpha}}{(1-|w|^{2})^{\frac{2+\alpha}{m}}}.
\end{equation}
By replacing $w$ with $\varphi_{z}(w)$, along with $\varphi_{z}(\varphi_{z}(w))=w$ and $k_{z}(\varphi_{z}(w))k_{z}(w)=1$, we obtain
\begin{equation*}
  |Sk_{z}(w)|
    \leq \frac{\|S_{z}1\|_{m,\alpha}}{(1-|\varphi_{z}(w)|^{2})^{\frac{2+\alpha}{m}}}|k_{z}(w)|.
\end{equation*}
A simple computation shows that
\begin{equation*}
  |SK_{z}(w)|\leq \|S_{z}1\|_{m,\alpha} \frac{(1-|z|^{2})^{-\frac{2+\alpha}{m}}(1-|w|^{2})^{-\frac{2+\alpha}{m}}}{|1-\overline{z}w|^{(1-\frac{2}{m})(2+\alpha)}}.
\end{equation*}
\end{proof}

\begin{theorem}\label{the1}
Suppose that $1<p<\infty$, $\alpha>-1$ and $S$ is a linear operator on $L^{p}_{a}(dA_{\alpha})$. If there exist a positive constant $C$ and some 
$$
m>p\frac{2+\alpha}{1+\alpha}\max\Big\{1,\frac{1}{p-1}\Big\},
$$
such that $\|S_{z}1\|_{m,\alpha}\leq C$ for all $z \in \mathbb{D}$, then $S$ is bounded.
\end{theorem}

\begin{proof}
From Lemma \ref{lemy2}, for $f\in L_{a}^{p}(dA_{\alpha})$ we have
\begin{equation}\label{formula4}
  f(w)=\sum_{k=1}^{\infty}c_{k}f_{k}(w),
\end{equation}
where $\{c_k\}\in l^p$,
\begin{eqnarray*}
  f_{k}(w)
    = (1-|a_{k}|^{2})^{\frac{(p-1)(2+\alpha)}{p}}K_{a_{k}}(w),
\end{eqnarray*}
and $\{a_{k}\} $ is an $r$-lattice in the Bergman metric.

By Lemma \ref{lemy16}, we have
\begin{equation}\label{formula25}
  \begin{split}
    |Sf_{k}(w)|
    &= |(1-|a_{k}|^{2})^{\frac{(p-1)(2+\alpha)}{p}}SK_{a_{k}}(w)|\\
    &  \leq \|S_{a_k}1\|_{m,\alpha}
    \frac{(1-|a_k|^{2})^{(\frac{p-1}{p}-\frac{1}{m})(2+\alpha)}(1-|w|^{2})^{-\frac{2+\alpha}{m}}}
    {|1-\overline{a_k}w|^{(1-\frac{2}{m})(2+\alpha)}}.\\
  \end{split}
\end{equation}
Applying (\ref{formula4}), (\ref{formula25}),  and H\"older's inequality, we get

\begin{equation}\label{formula6}
  \begin{split}
    \int_{\mathbb{D}}
    &|Sf(w)|^{p}
    dA_{\alpha}(w)\\
    & \leq \int_{\mathbb{D}}\bigg(\sum_{k=1}^{\infty}
    \Big|c_{k} \|S_{a_{k}}1\|_{m,\alpha} \frac{(1-|a_{k}|^{2})^{(\frac{p-1}{p}-\frac{1}{m})(2+\alpha)}(1-|w|^{2})^{-\frac{2+\alpha}{m}}}
    {|1-\overline{a_{k}}w|^{(1-\frac{2}{m})(2+\alpha)}}\Big|\bigg)^{p}\, dA_{\alpha}(w)\\
    &\leq \int_{\mathbb{D}}\bigg(\sum_{k=1}^{\infty}
    \Big|c_{k} \|S_{a_{k}}1\|_{m,\alpha} \frac{(1-|a_{k}|^{2})^{(\frac{p-1}{p}-\frac{1}{m})(2+\alpha)-n_{1}}(1-|w|^{2})^{-\frac{2+\alpha}{m}}}
    {|1-\overline{a_{k}}w|^{(1-\frac{2}{m})(2+\alpha)-n_{2}}}\Big|^{p}\bigg)\\
    &\quad\times \bigg(\sum_{k=1}^{\infty} \frac{(1-|a_{k}|^{2})^{qn_{1}} }{|1-\overline{a_{k}}w|^{qn_{2}}}\bigg)^{\frac{p}{q}}\, dA_{\alpha}(w),\\
  \end{split}
\end{equation}
where $0<n_1,n_2<\infty$ and $q=p/(p-1)$.

To get the desired result, we will prove that the last term in (\ref{formula6}) can be controlled by $\sum_{k=1}^{\infty}|c_{k}|^{p} \|S_{a_{k}}1\|_{m,\alpha}^{p}$.
Next, the assumption gives
\begin{equation}\nonumber
  \frac{p-1}{p}< \Big(\frac{p-1}{p}-\frac{1}{m}\Big)(2+\alpha),~~ \frac{2+\alpha}{m}<\frac{1+\alpha}{p}.
\end{equation}
Choose positive numbers $n_1$ and $n_2$ satisfying
 $$\frac{p-1}{p}<n_{1}<\Big(\frac{p-1}{p}-\frac{1}{m}\Big)(2+\alpha), \quad n_{1}<n_{2}<n_{1}+\frac{1+\alpha}{p}-\frac{2+\alpha}{m},$$
 which implies that
 $$q_{n_{2}} > qn_{1}>q \frac{p-1}{p}=1.$$

 By Lemma \ref{ly1}, there exists a positive constant $C_{q,n_{1},n_{2},r}$ depending on $q,n_{1},n_{2},r$ such that
 \begin{equation}\label{f3}
   \sum_{k=1}^{\infty} \frac{(1-|a_{k}|^{2})^{qn_{1}} }{|1-\overline{a_{k}}w|^{qn_{2}}} \leq \frac{C_{q,n_{1},n_{2},r}}{(1-|w|^{2})^{q(n_{2}-n_{1})}}.
 \end{equation}
Furthermore, we have
\begin{equation}\nonumber
   \Big(-\frac{2+\alpha}{m}-n_{2}+n_{1}\Big)p+\alpha>-1.
\end{equation}
By Lemma \ref{lemy15},
\begin{equation}\label{formula16}
  \begin{split}
   \int_{\mathbb{D}} \Big|\frac{(1-|w|^{2})^{-\frac{2+\alpha}{m}-n_{2}+n_{1}}}
    {|1-\overline{a_{k}}w|^{(1-\frac{2}{m})(2+\alpha)-n_{2}}}\Big|^{p}\, dA_{\alpha}(w)
    \leq  \frac{C_{m,p,n_{1},n_{2}}}{(1-|a_{k}|^{2})^{(p-1)(2+\alpha)-\frac{p(2+\alpha)}{m}-pn_{1}}}.
  \end{split}
\end{equation}

Combining $(\ref{formula6})$, $(\ref{f3})$, $(\ref{formula16})$ and Lebesgue's Monotone Convergence Theorem, we get
\begin{equation}\nonumber
  \begin{split}
   \int_{\mathbb{D}} |Sf(w)|^{p}
    dA_{\alpha}(w)
    \leq C_{m,p,r,n_{1},n_{2}} \sum_{k=1}^{\infty}|c_{k}|^{p}\|S_{a_{k}}1\|_{m,\alpha}^{p}.
  \end{split}
\end{equation}
Since $\|S_{z}1\|_{m,\alpha} \leq C$ for all $z \in \mathbb{D}$, applying Lemma \ref{lem2} yields
$$\|Sf\|_{p,\alpha} \leq K \|f\|_{p,\alpha}$$
for some positive constant $K$.
This implies that $S$ is bounded on $L_{a}^{p}(dA_{\alpha})$.
\end{proof}

We next characterize the boundedness of linear operators on $L_a^p(dA_\alpha)$ for $0<p\leq1$. Although $\|\cdot\|_{p,\alpha}$ is not a norm on $L_a^p(dA_\alpha)$ when $0<p<1$, the boundedness of $S$ is defined as $\|Sf\|_{p,\alpha} \leq C \|f\|_{p,\alpha}$.
For any $\delta>0$, then $\beta=(2+\alpha)/p-2+\delta>-1$. Let $b=\beta+2$, then $b>\max\{1,1/p\}+(1+\alpha)/p$
and $(pb-2-\alpha)/p=\delta$. Therefore, by Lemma \ref{lemy2}, for $f\in L_a^p(dA_\alpha)~(0<p\leq 1)$, we have
\begin{equation}\label{f1}
  f(w)=\sum_{k=1}^{\infty}c_{k}\frac{(1-|a_k|^{2})^{\delta}}{(1-\overline{a_k}w)^{2+\beta}},
\end{equation}
where $\{c_k\}\in l^p$ and $\{a_k\}$ is an $r$-lattice on $\mathbb{D}$.

\begin{theorem}\label{the6}
Suppose that $0<p\leq 1$, $\alpha>-1$, $\delta>0$, $\beta=\frac{2+\alpha}{p}-2+\delta$ and $S$ is a linear operator on $L^{p}_{a}(dA_{\alpha})$. If there exist a positive constant $C$ and some
$$m>\max\left\{\frac{2+\alpha}{p\delta}+1,\frac{1+p\delta}{1+\alpha}+1\right\},$$
 such that $\|S_{z}1\|_{m,\beta}\leq C$ for all $z \in \mathbb{D}$, then $S$ is bounded.
\end{theorem}
\begin{proof}
Let
$$K_{z}(w)=\frac{1}{(1-\overline{z}w)^{2+\beta}} ~\text{and}~f_{k}(w)=(1-|a_{k}|^{2})^{\delta}K_{a_{k}}(w).$$
By Lemma \ref{lemy2}, $(\ref{f1})$ can be written as
$$f(w)=\sum_{k=1}^{\infty}c_{k}f_{k}(w).$$
Applying Lemma \ref{lemy16},
\begin{equation}\label{formula24}
  \begin{split}
    |Sf(w)|^{p} \leq \sum_{k=1}^{\infty}|c_{k}|^{p} \|S_{a_{k}}1\|_{m,\beta}^{p} \frac{(1-|a_{k}|^{2})^{p(\delta-\frac{2+\beta}{m})}(1-|w|^{2})^{-p\frac{2+\beta}{m}}}{|1-\overline{a_{k}}w|^{p(1-\frac{2}{m})
    (2+\beta)}}.
  \end{split}
\end{equation}
It is easy to check that $-p\frac{2+\beta}{m}+\alpha>-1$ and
$$c= p(1-\frac{1}{m})(2+\beta)-(2+\alpha)=p(\delta-\frac{2+\beta}{m})>0.$$

By Lemma \ref{lemy15}, there exists positive constant $C_{m,p,\delta,r}$ depending on $m$, $p$ and $\delta$ such that
\begin{equation}\nonumber
  \begin{split}
   \int_{\mathbb{D}} \frac{(1-|w|^{2})^{-p\frac{2+\beta}{m}}}{|1-\overline{a_k}w|^{p(1-\frac{2}{m})
    (2+\beta)}}\, dA_{\alpha}(w) \leq \frac{C_{m,p,\delta,r}}{(1-|a_k|^{2})^{c}}.
  \end{split}
\end{equation}
Combining this with (\ref{formula24}), (\ref{formula10}), Lemma \ref{lemy2},  and Lebesgue's Monotone Convergence Theorem yields
\begin{equation}\nonumber
  \begin{split}
  \|Sf\|_{p,\alpha}^{p}
   \leq C_{m,p,\delta,r} \sum_{k=1}^{\infty}|c_{k}|^{p} \|S_{a_{k}}1\|_{m,\beta}^{p}\leq  C^{p} C_{m,p,\delta,r} \|f\|^{p}_{p,\alpha}.
  \end{split}
\end{equation}
Therefore, $S$ is bounded on $L_{a}^{p}(dA_{\alpha})$.
\end{proof}

If $p\delta<1+\alpha$, then
$$\frac{2+\alpha}{p\delta}+1>2+\frac{1}{\alpha+1}>\frac{1+p\delta}{1+\alpha}+1.$$
Otherwise, $p\delta>1+\alpha$ indicates $$\frac{2+\alpha}{p\delta}+1<2+\frac{1}{\alpha+1}<\frac{1+p\delta}{1+\alpha}+1.$$
In addition, $\|\cdot\|_{p,\alpha}$  increases with respect to $p$ as $0<p<\infty$.
If $\delta=\frac{1+\alpha}{p}$, then
$$\frac{2+\alpha}{p\delta}+1=2+\frac{1}{\alpha+1}=\frac{1+p\delta}{1+\alpha}+1 .$$
Thus, we can obtain the sharp result as follows.

\begin{corollary}
Suppose that $0<p\leq 1$, $\alpha>-1$, $\beta=\frac{3+2\alpha}{p}-2$ and $S$ is a linear operator on $L^{p}_{a}(dA_{\alpha})$. If there exist a positive constant $C$ and some $m>2+\frac{1}{1+\alpha}$ such that $\|S_{z}1\|_{m,\beta}\leq C$ for all $z \in \mathbb{D}$, then $S$ is bounded.
\end{corollary}

If $1<p\leq 2$, then 
$$
\frac{p(2+\alpha)}{(p-1)(1+\alpha)} >\frac{p(2+\alpha)}{1+\alpha}.
$$
In this case, if there exist a positive constant $C$ and some $m>\frac{p(2+\alpha)}{(p-1)(1+\alpha)}$, such that $\|S_{z}1\|_{m,\alpha} \leq C $ for all $z\in \mathbb{D}$, by Theorem \ref{the1}, $S$ is bounded on $L_a^p(dA_\alpha)$. This yields that $S$ is bounded from $L_a^p(dA_\alpha)$ to $L_a^q(dA_\alpha)$ as $0<q\leq p$. If $2<p<\infty$, then 
$$
\frac{p(2+\alpha)}{(p-1)(1+\alpha)} <\frac{p(2+\alpha)}{1+\alpha}.
$$

For $m$ such that
$$
\frac{p(2+\alpha)}{(p-1)(1+\alpha)}<m\leq \frac{p(2+\alpha)}{1+\alpha},
$$ the condition $\|S_{z}1\|_{m,\alpha} \leq C $ for all $z\in \mathbb{D}$   does not necessarily imply that $S$ is bounded on
$L_a^p(dA_\alpha)$. In this case, we have the following result:



\begin{theorem}\label{the4}
Suppose $2<p<\infty$, $\alpha>-1$ and $S$ is a linear operator from $L_a^p(dA_\alpha)$ to $L_a^q(dA_\alpha)$, and there exist a positive constant $C$ and some $m$ with 
$$
\frac{p(2+\alpha)}{(p-1)(1+\alpha)}<m\leq \frac{p(2+\alpha)}{1+\alpha}$$
such that $\|S_{z}1\|_{m,\alpha} \leq C $ for all $z\in \mathbb{D}$.
\begin{itemize}
  \item [(a)] If $p\geq m$ and $0<q<m\frac{1+\alpha}{2+\alpha}$, then $S$ is a bounded linear operator from $L_a^p(dA_\alpha)$ to $L_a^q(dA_\alpha)$.
  \item [(b)] If $p< m$ and $0<q<\frac{p}{2+\alpha}$, then $S$ is a bounded linear operator from $L_a^p(dA_\alpha)$ to $L_a^q(dA_\alpha)$.
\end{itemize}
\end{theorem}
\begin{proof}
By Lemma \ref{lemy2}, for $f \in L_{a}^{p}(dA_{\alpha})$,
\begin{equation}\nonumber
  f(w)=\sum_{k=1}^{\infty}c_{k}f_k(w),
\end{equation}
where $$f_k(w)=\frac{(1-|a_{k}|^{2})^{\frac{(p-1)(2+\alpha)}{p}}}{(1-w\overline{a_{k}})^{2+\alpha}},$$
$\{c_k\}\in l^p$,  and $\{a_k\}$ is an $r$-lattice on $\mathbb{D}$. As in the proof of Theorem \ref{the1}, we have the estimates
\begin{equation}\nonumber
  \begin{split}
    |Sf_{k}(w)|\leq \|S_{a_{k}}1\|_{m,\alpha} \frac{(1-|a_{k}|^{2})^{(\frac{p-1}{p}-\frac{1}{m})(2+\alpha)}(1-|w|^{2})^{-\frac{2+\alpha}{m}}}
    {|1-\overline{a_{k}}w|^{(1-\frac{2}{m})(2+\alpha)}}
  \end{split}
\end{equation}
and
\begin{equation}\label{formula21}
  \begin{split}
    \int_{\mathbb{D}}|Sf(w)|^{q}\, dA_{\alpha}(w)
    &\leq \int_{\mathbb{D}} (1-|w|^{2})^{-(2+\alpha)\frac{q}{m}}\Big(\sum_{k=1}^{\infty}
    |c_{k}|^{p} \|S_{a_{k}}1\|_{m,\alpha}^{p} \Big)^{\frac{q}{p}}\\
    &\quad\times \Big(\sum_{k=1}^{\infty}
    \frac{(1-|a_{k}|^{2})^{p'(2+\alpha)(\frac{p-1}{p}-\frac{1}{m})}}
    {|1-\overline{a_{k}}w|^{p'(1-\frac{2}{m})(2+\alpha)}}\Big)^{\frac{q}{p'}}\,  dA_{\alpha}(w) ,\\
  \end{split}
\end{equation}
where $\frac{1}{p}+\frac{1}{p'}=1$.
Define a function $L(w)$ as
\begin{equation}\nonumber
  \begin{split}
    L(w)=\sum_{k=1}^{\infty}
    \frac{(1-|a_{k}|^{2})^{p'(2+\alpha)(\frac{p-1}{p}-\frac{1}{m})}}
    {|1-\overline{a_{k}}w|^{p'(2+\alpha)(1-\frac{2}{m})}}.
  \end{split}
\end{equation}
By the assumption,
$$ p'(2+\alpha)\Big(\frac{p-1}{p}-\frac{1}{m}\Big) >p'(2+\alpha) \bigg( \frac{p-1}{p}- \frac{(p-1)(1+\alpha)}{p(2+\alpha)}\bigg)=1.$$

If $m>2$, then $(1-2/m)>0$ and
\begin{equation}\label{f4}
  p'(2+\alpha)\Big(1-\frac{2}{m}\Big)-p'(2+\alpha)\Big(\frac{p-1}{p}-\frac{1}{m}\Big)=p'(2+\alpha)\Big(\frac{1}{p}-\frac{1}{m}\Big) .
\end{equation}
 The following arguments are divided into three cases: (i) $p>m$; (ii) $p<m$; (iii) $p=m$. Applying Lemma \ref{ly1}, we find a positive constant $C_{p,m,r}$ dependent on $p,m,r$ such that the following statements hold:

 (i) If $p>m$, $(\ref{f4})$ shows that
 $$p'\Big(1-\frac{2}{m}\Big)(2+\alpha) < p'(2+\alpha)\Big(\frac{p-1}{p}-\frac{1}{m}\Big).$$
 By Lemma \ref{ly1}, $L(w) \leq C_{p,m,r}$. In this case,
 $$  - \frac{q}{m}(2+\alpha)>-1,$$
 which yields that
\begin{eqnarray*}
      \int_{\mathbb{D}}|Sf(w)|^{q}\, dA_{\alpha}(w)
      &\leq&  C_{p,m,r}^{\frac{q}{p'}}\int_{\mathbb{D}} (1-|w|^{2})^{-\frac{q}{m}(2+\alpha)}\\
      && \times       \Big(\sum_{k=1}^{\infty}
    |c_{k}|^{p} \|S_{a_{k}}1\|_{m,\alpha}^{p} \Big)^{\frac{q}{p}}\, dA_{\alpha}(w)\\
    &\leq& C_{p,q,m,r}\Big(\sum_{k=1}^{\infty}
    |c_{k}|^{p} \|S_{a_{k}}1\|_{m,\alpha}^{p} \Big)^{\frac{q}{p}}.
\end{eqnarray*}

(ii) If $p<m$, $(\ref{f4})$ shows that
$$p'\Big(1-\frac{2}{m}\Big)(2+\alpha) > p'(2+\alpha)\Big(\frac{p-1}{p}-\frac{1}{m}\Big),$$
 By Lemma \ref{ly1}, $$L(w)\leq  \frac{C_{p,m,r}}{(1-|w|^{2})^{c}},$$
  where $c=p'(\frac{1}{p}-\frac{1}{m})(2+\alpha)$. Furthermore
  $$  -q\Big(\frac{1}{p}-\frac{1}{m}\Big)(2+\alpha)-\frac{q}{m}(2+\alpha) >-1 $$
  which implies that
  \begin{equation}\nonumber
  \begin{split}
    \int_{\mathbb{D}}|Sf(w)|^{q}\, dA_{\alpha}(w)
    &\leq C_{p,m,r}^{\frac{q}{p'}}\int_{\mathbb{D}}  (1-|w|^{2})^{-q(\frac{1}{p}-\frac{1}{m})(2+\alpha)-\frac{q}{m}(2+\alpha)}\\
    &\ \ \ \  \times \Big(\sum_{k=1}^{\infty}
    |c_{k}|^{p} \|S_{a_{k}}1\|_{m,\alpha}^{p} \Big)^{\frac{q}{p}} \,  dA_{\alpha}(w)\\
     &\leq C_{p,q,m,r}\Big(\sum_{k=1}^{\infty}
    |c_{k}|^{p} \|S_{a_{k}}1\|_{m,\alpha}^{p} \Big)^{\frac{q}{p}}.
  \end{split}
\end{equation}

(iii) If $p=m$, $(\ref{f4})$ shows that
$$p'\Big(1-\frac{2}{m}\Big)(2+\alpha) = p'(2+\alpha)\Big(\frac{p-1}{p}-\frac{1}{m}\Big).$$
 By Lemma \ref{ly1}, $$L(w)\leq \frac{C_{p,m,r}}{|w|^{2}}\log\frac{1}{1-|w|^{2}} .$$
 It is clear that
 $$ \lim_{|w| \rightarrow 0} \frac{1}{|w|^{2}}\log\frac{1}{1-|w|^{2}}=1~\text{and}~\lim_{|w| \rightarrow 1^{-}} (1-|w|^{2})^{t}\log\frac{1}{1-|w|^{2}}=1$$
 for all $0<t<1$.
  Therefore,
\begin{equation}\nonumber
  \begin{split}
    \int_{\mathbb{D}}|Sf(w)|^{q}\, dA_{\alpha}
    &\leq C_{p,m,r}^{\frac{q}{p'}}\int_{\mathbb{D}} \Big( \frac{1}{|w|^{2}}\log\frac{1}{1-|w|^{2}}\Big)^{\frac{q}{p'}}(1-|w|^{2})^{-\frac{q}{m}(2+\alpha)}\\
    &\ \ \ \ \times \Big(\sum_{k=1}^{\infty}
    |c_{k}|^{p} \|S_{a_{k}}1\|_{m,\alpha}^{p} \Big)^{\frac{q}{p}}\, dA_{\alpha}(w)\\
     &\leq C_{p,q,m,r}\Big(\sum_{k=1}^{\infty}
    |c_{k}|^{p} \|S_{a_{k}}1\|_{m,\alpha}^{p} \Big)^{\frac{q}{p}}.
  \end{split}
\end{equation}

If $m\leq 2$, then $p'(1-\frac{2}{m})(2+\alpha)<0<p'(2+\alpha)(\frac{p-1}{p}-\frac{1}{m}) $. By Lemma \ref{ly1}, there exists a positive constant $C_{p,m,r}$ dependent on $p,m,r$ such that $L(w) \leq C_{p.m,r}$. Since $p>2$, in this case, $p>m$. The rest of the proof is similar to the proof of (i).

Since $\|S_{z}1\|_{m,\alpha} \leq C $, combining above arguments and Lemma \ref{lem2}, we get
 $$\|Sf\|_{q,\alpha} \leq  C C_{p,q,m}^{\frac{1}{q}}\|f\|_{p,\alpha}.$$
  The proof is complete.

\end{proof}

\section{Compact linear operator}

In this section, we aim to characterize the compactness of a linear operator $S$ on $L^{p}_{a}(dA_{\alpha})$.
Recall that if $S$ is a compact linear operator on $L_{a}^{2}(dA_{\alpha})$, then $\widetilde{S}(z) \rightarrow 0$ as $z \rightarrow \partial \mathbb{D}$ (see \cite{zhu}).
In \cite[Lemma 3.10]{Axler1988}, Axler proved that $\|K_{z}\|_{p}$ are equivalent to $(1-|z|^{2})^{-\frac{2(p-1)}{p}}$ for all $z\in \mathbb{D}$. Zeng \cite[Lemma 2.3]{zeng} has showed that $\frac{K_{z}}{\|K_{z}\|_{p}} \rightarrow 0$ weakly in $L^{p}_{a}(dA)$ as $z \rightarrow \partial \mathbb{D}$ for $1<p<\infty$.

Note that by the definition of $\widetilde{S}$, $\widetilde{S}(z)$ is equivalent to
$$
\bigg\langle S\Big(\frac{K_{z}}{\|K_{z}\|_{p}}\Big),\frac{K_{z}}{\|K_{z}\|_{q}} \bigg\rangle,
$$
where $1<p<\infty$ and $\frac{1}{p}+\frac{1}{q}=1$. Therefore, the relationship between compact operators and Berezin transform can be extended to $L^{p}_{a}(dA)\ (1<p<\infty)$.
That is, if $S$ is compact on $L_{a}^{p}(dA)~(1<p<\infty)$, then $\widetilde{S}(z) \rightarrow 0$ as $z \rightarrow \partial \mathbb{D}$. By Lemmas \ref{lemy15} and \ref{lemy9}, we can easily establish the following lemma, extending the results from \cite{Axler1988} and \cite{zeng} to weighted Bergman spaces $L^{p}_{a}(dA_{\alpha})$.

\begin{lemma}\label{lemy3}
Suppose $1<p<\infty$ and $\alpha >-1$.
\begin{itemize}
\item [(a)] Let $c=(p-1)(2+\alpha)$, then
$$
\frac{1}{(1-|z|^{2})^{\frac{p-1}{p}(2+\alpha)}} \leq \|K_{z}\|_{p,\alpha} \leq \bigg(  \frac{\Gamma(2+\alpha)\Gamma(c)}{\Gamma(\frac{2+\alpha+c}{2})^{2}(1-|z|^{2})^{(p-1)(2+\alpha)}}\bigg)^{\frac{1}{p}}.
$$
  \item [(b)]  $\frac{K_{z}}{\|K_{z}\|_{p,\alpha}} \rightarrow 0$ weakly in $L_{a}^{p}(dA_{\alpha})$ as $z \rightarrow \partial \mathbb{D}$.
\end{itemize}
\end{lemma}
By above lemma, it is not hard to check that for each compact operator $S$ on $L^{p}_{a}(dA_{\alpha})(1<p<\infty)$, we have $\widetilde{S}(z) \rightarrow 0$ as $z \rightarrow \partial \mathbb{D}$.

\begin{lemma} \cite[Lemma 5.3]{Jie}\label{lemy1}
Suppose that $1<p<\infty$, $S$ is a bounded operator on $L_{a}^{p}(dA)$  such that
$$\sup_{z\in \mathbb{D}} \|S_{z}1\|_{m} <\infty$$
for some $m>1$, then $\widetilde{S}(z) \rightarrow 0$ as $z \rightarrow \partial \mathbb{D}$ if and only if for every $t\in [1,m)$, $\|S_{z}1\|_{t} \rightarrow 0$ as $z \rightarrow \partial \mathbb{D}$.

\end{lemma}

By Lemma \ref{lemy15} and the method in Lemma \ref{lemy1}, it is not hard to check that the result in Lemma \ref{lemy1} also holds on  Bergman space $L^{p}_{a}(dA_{\alpha})$.

\begin{theorem}\label{the3}
Suppose that $1<p<\infty$, $\alpha>-1$ and $S$ is a linear operator on $L_{a}^{p}(dA_{\alpha})$, if there exist a positive constant $C$ and some $m>p\frac{2+\alpha}{1+\alpha} \max\big\{1,\frac{1}{p-1}\big\}$, such that $\|S_{z}1\|_{m,\alpha}\leq C$ for all $z \in \mathbb{D}$, then $S$ is compact if and only if $\widetilde{S}(z) \rightarrow 0$ as $z \rightarrow \partial \mathbb{D}$.
\end{theorem}

\begin{proof}
By Theorem \ref{the1}, the assumption yields that $S$ is bounded $L_{a}^{p}(dA_{\alpha})$. If $S$ is compact, by Lemma \ref{lemy3},  $\widetilde{S}(z) \rightarrow 0$ as $z \rightarrow \partial \mathbb{D}$.

If $\widetilde{S}(z) \rightarrow 0$ as $z \rightarrow \partial \mathbb{D}$, then by Lemma \ref{lemy1},
to obtain the desired result, we only need to show that if $\|S_{z}1\|_{m,\alpha}\rightarrow 0$ as $z\rightarrow \partial\mathbb{D}$, then $S$ is compact.

Suppose $f_{n}\rightarrow 0$ weakly in $L^{p}_{a}(dA_{\alpha})$ as $n \rightarrow \infty$. To show that $S$ is compact, it is sufficient to prove that $\|Sf_{n}\|_{p,\alpha}\rightarrow 0$ as $n\rightarrow \infty$.
Lemma \ref{lemy2} gives that
$$f_{n}(w)=\sum_{k=1}^{\infty} c_{k,n}\frac{(1-|a_{k}|^{2})^{\frac{(p-1)(2+\alpha)}{p}}}{(1-w\overline{a_{k}})^{2+\alpha}},$$
where $\{c_{k,n}\}_{k\geq1}\in l^p$ and $\{a_k\}$ is an $r$-lattice on $\mathbb{D}$.

By the proof in Theorem \ref{the1}, we know that
\begin{equation}\nonumber
  \|Sf_n\|_{p,\alpha}^{p} \leq C\sum_{k=1}^{\infty} |c_{k,n}|^{p}\|S_{a_{k}}1\|_{m,\alpha}^{p},
\end{equation}
with positive constant $C$.
Since $f_{n}\rightarrow 0$ weakly, Lemma \ref{lem2} shows that for any fixed $0<R<1$
$$\sum_{|a_{k}|\leq R} |c_{k,n}|^{p} \rightarrow 0~~\text{as}~~ n \rightarrow \infty,$$
i.e., for any $\varepsilon >0$, there exists $N>0$, if $n>N$, then
\begin{equation}\label{f5}
  \sum_{|a_{k}|\leq R} |c_{k,n}|^{p} <\varepsilon
\end{equation}
 Since $\|S_{z}1\|_{m,\alpha}\rightarrow 0 $ as $z\rightarrow \partial \mathbb{D}$, then for the $\varepsilon$ in $(\ref{f5})$, there exists $R<1$, such that
$$\|S_{z}1\|_{m,\alpha} <\varepsilon, \text{if }|z|>R .$$
This means that
$$ \sum_{|a_{k}|>R} |c_{k,n}|^{p}\|S_{a_{k}}1\|_{m,\alpha}^{p} \leq \|f_{n}\|_{p}^{p} \varepsilon^{p} \leq C_{f}^{p} \varepsilon^{p},$$
where $C_{f} =\sup_{n} \|f_{n}\|_{p}$ is bounded, because Lemma \ref{lemy9} gives that $\{f_{n}\}$ is bounded in $L^{p}_{a}(dA_{\alpha})$. Therefore,
\begin{equation}\nonumber
  \begin{split}
    \sum_{k=1}^{\infty} |c_{k,n}|^{p}\|S_{a_{k}}1\|_{m,\alpha}^{p}
    &=\sum_{|a_{k}|\leq R} |c_{k,n}|^{p}\|S_{a_{k}}1\|_{m,\alpha}^{p}+ \sum_{|a_{k}|>R} |c_{k,n}|^{p}\|S_{a_{k}}1\|_{m,\alpha}^{p}\\
    &\leq C^{p} \varepsilon +C_{f}^{p} \varepsilon^{p},
  \end{split}
\end{equation}
which shows that  $\|Sf_{n}\|_{p,\alpha}\rightarrow 0$ as $n\rightarrow 0$.
The proof is complete.

\end{proof}

\end{document}